\documentclass[12pt]{article}

\usepackage{amsmath,amsfonts,amssymb,amsthm,epsfig,epstopdf,url,array,color}
\usepackage[T1]{fontenc}
\usepackage[utf8]{inputenc}
\usepackage{authblk}
\usepackage[sort,numbers,square,comma]{natbib}
\theoremstyle{plain}
\newtheorem{theorem}{Theorem}
\newtheorem{lemma}{Lemma}
\newtheorem{proposition}{Proposition}
\newtheorem*{corollory}{Corollary}

\theoremstyle{definition}
\newtheorem{defn}{Definition}
\newtheorem*{conj}{Conjecture}

\theoremstyle{remark}

\title{On continuous movement of the discrete spectrum of Schr\"odinger operators}
\author[1]{M. N. N. Namboodiri}
\author[2]{S. Satheesh Kumar}
\affil[1]{Department of Mathematics, Cochin University of Science and Technology, Kochi - 22, India}
\affil[2]{Naval Physical and Oceanographic Laboratory, Kochi - 21, India}
\date{}

\begin{document}
\maketitle

\begin{abstract}
Continuous movement of discrete spectrum of the Schr\"{o}dinger operator $H(z)=-\frac{d^2} {dx^2}+V_0+z V_1$, with $\int_0^\infty {x |V_j(x)| dx} < \infty$, on the half-line is studied as $z$ moves along a continuous path in the complex plane. The analysis provides information regarding the members of the discrete spectrum of the non-selfadjoint operator that are evolved from the discrete spectrum of the corresponding selfadjoint operator.
\end{abstract}

\section{Introduction}\label{Ch2:Sec:Introduction}
We consider the operator valued analytic function
\begin{equation}
H(z)=-{\frac{d^2}{dx^2}}+V_0 + z V_1
\end{equation}
defined on the complex plane $\mathbb{C}$, where $V_0$, $V_1$ are real valued bounded measurable functions vanishing sufficiently rapidly as $|x| \rightarrow \infty$. $H(z)$ is a Schr\"odinger operator on $L^2(0,\infty)$ with domain $\text{Dom}(H(z))=\{f \in L^2(0,\infty): f', f'' \in L^2(0,\infty), f(0)=0\}$. In this work, continuous evolution of the discrete spectrum of $H(z)$ is studied as $z$ varies along a continuous path in $\mathbb{C}$. In particular, as $z$ varies along the imaginary line from $0$ to $i$, the evolution of the discrete spectrum of the self-adjoint operator $H(0)=-{\frac{d^2}{dx^2}}+V_0$ is of special interest to us.

The distribution of the discrete spectrum of self-adjoint Schr\"odinger operator has been studied extensively, the same is not the case with non-selfadjoint Schr\"odinger operator. The spectral theorem and min-max principles for the selfadjoint case play major role in its theoretical development, whereas such tools are not available for non-selfadjoint operators. Each non-selfadjoint problem needs to be studied separately. Our effort to extract some information on the discrete spectrum of the non-selfadjoint Schr\"odinger operator
\begin{equation}\label{nonSelfEqn}
H(i)=-{\frac{d^2}{dx^2}}+V_0+iV_1
\end{equation}
using the discrete spectrum of corresponding self-adjoint Schr\"odinger operator
\begin{equation}\label{selfEqn}
H(0)=-{\frac{d^2}{dx^2}}+V_0
\end{equation}
is quite a different approach and we intend to prove the following result.
\setcounter{theorem}{1}
\begin{theorem}\label{labMainTheorem}
If $V_0$ and $V_1$ are such that $\int_{0}^{\infty}{x |V_j(x)| dx} < \infty$, $j=0,1$ and let $\kappa_1$ be in the discrete spectrum of $H(i)$ then there exist $(1)$ $t_0$ $(0\le t_0<1)$, $(2)$ a real number $\kappa_0$, a member in the discrete spectrum or a spectral singularity of the operator $H(i t_0)$ and $(3)$ a continuous path $\kappa(t)$ such that $\kappa(0)=\kappa_0$, $\kappa(1)=\kappa_1$ and each $\kappa(t)$, $0<t \le 1$, is a discrete eigenvalue of the operator $H(i(t_0+t))$.
\end{theorem}

\section{Potentials with Compact Support}\label{Ch2:Sec:Preliminaries}
We assume initially, that $V_0$ and $V_1$ are bounded, continuous (except for jump discontinuity) and compactly supported real functions, then it follows from \cite{Abramov2001} that the spectrum of the operator $H(z)$ consists of the essential spectrum $\sigma_{ess}(H)=[0,\infty)$ and a finite number of discrete eigenvalues. Further $H(z)$ is an operator valued analytic function (an analytic function of type (A)) and hence from \cite[p. 370]{Kato1980}, the finite system of eigenvalues of $H(z_0)$ are branches of one or several analytic functions that have at most algebraic singularities near $z=z_0$. Also it follows from \cite{Kato1980} that if $\kappa(z)=-\lambda^2(z)$, with $\text{Re}(\lambda(z)) > 0$, is a member of the discrete spectrum of $H(z)$ and is analytic at $z$, then its derivative
\begin{equation}
\kappa'(z)=\frac{\left\langle V_1 \phi(z,\cdot),\bar{\phi}(z,\cdot) \right\rangle}{\left\langle \phi(z,\cdot), \bar{\phi}(z,\cdot) \right\rangle}=\frac{\int_0^\infty {V_1(x)\phi^2(z,x) dx}}{\int_0^\infty {\phi^2(z,x) dx}}
\end{equation}
where $\phi(z,\cdot)$ is the normalized eigenfunction of $H(z)$ corresponding to $\kappa(z)$.

Before proceeding, it is desirable to state few results which are useful in our discussion.

\begin{lemma}\label{lemma1}
\cite[p. 136]{Sanchez1979} Let the differential equation $\mathbf{x}'=f(t,\mathbf{x},\mathbf{\lambda)}$, $t$ a scalar variable, $\mathbf{x}=(x_1,x_2,\ldots,x_n)$ and $\mathbf{\lambda}=(\lambda_1,\lambda_2,\ldots,\lambda_\nu)$ be given, where $f(t,\mathbf{x},\mathbf{\lambda)}$ and $\partial f/\partial x_i$ are defined and continuous in some domain $B$ contained in $\mathbb{R}^{n+\nu+1}$. If $(t_0,\mathbf{x}_0,\mathbf{\lambda}_0)$ belongs to $B$, then there exist positive numbers $r$ and $p$ such that
\begin{enumerate}
\item{Given any $\mathbf{\lambda}$ such that $\|\mathbf{\lambda}-\mathbf{\lambda}_0 \| \le p$, there exists a unique solution $\mathbf{x}=\mathbf{x}(t,\mathbf{\lambda})$ of the given differential equation, defined for $|t-t_0| \le r$ and satisfying $\mathbf{x}(t_0,\mathbf{\lambda})=\mathbf{x}_0$.}
\item{The solution $\mathbf{x}=\mathbf{x}(t,\lambda)$ is a continuous function of $t$ and $\lambda$}
\end{enumerate}
\end{lemma}
\begin{lemma}\label{lemma2}
\cite[p. 169]{Hirsch1974} Let $W \subset E$ be open in $\mathbb{R}^n$ and suppose $f:W \rightarrow E$ has Lipschitz constant $K$. Let $\mathbf{y}(t)$, $\mathbf{z}(t)$ be solutions to
\begin{equation}
\mathbf{x}'=f(\mathbf{x})
\end{equation}
on the closed interval $[t_0,t_1]$. Then for all $t \in [t_0,t_1]:$
$$|\mathbf{y}(t)-\mathbf{z}(t)| \le |\mathbf{y}(t_0)-\mathbf{z}(t_0)|\exp(K(t-t_0))$$
\end{lemma}
Lemma~\ref{lemma1} talks about the continuity of the solution of differential equation with respect to the coefficient parameters, and Lemma~\ref{lemma2} is about the continuity with respect to the initial conditions. Combining both we will have the following result for a second order linear differential equation.
\begin{lemma}\label{lemma3}
Let the second order linear differential equation $x''+p_n(t) x'+q_n(t)x=r_n(t)$ be given. $p_n(t)$, $q_n(t)$, and $r_n(t)$ be continuous on $[a,b]$ and $p_n \rightarrow p$, $q_n \rightarrow q$, and $r_n \rightarrow r$ uniformly on $[a,b]$. Let $x_n(t)$ be the solution of the differential equation on $[a,b]$ satisfying the initial conditions: $x_n(a)=\alpha_n$, $x'_n(a)=\beta_n$. Also assume that $\alpha_n \rightarrow \alpha$ and $\beta_n \rightarrow \beta$. Then $x_n \rightarrow x$ uniformly on $[a,b]$, where $x(t)$ is the solution of the differential equation satisfying $x(a)=\alpha$, $x'(a)=\beta$.
\end{lemma}

\section{Evolution of the Discrete Spectrum}\label{Ch2:Sec:Evolution}

Consider any path in the complex plane $\mathbb{C}$ traced by $z$ starting from $0$, then each of the discrete spectrum element of $H(z)$ moves in the complex plane until it ceases to exist as a discrete spectrum member. And it follows from \cite{Kato1980} that this movement is analytic except for isolated points of algebraic singularities. If $z$ varies along the real line, then $H(z)$ is a family of self-adjoint operators and the discrete eigenvalues, if exist, move on the negative real axis and all are simple (see Lemma~\ref{Ch2:Lemma:Multiplicity}). Let $\kappa_0$ be an eigenvalue of $H(0)$ and as $z$ varies from $0$ to $\infty$ over the positive real line, the eigenvalue $\kappa(z)$ starts moving continuously (analytically) from $\kappa_0$ and if we further assume $V_1$ is positive on its support then at some $\zeta>0$ the path traced by $\kappa(z)$ terminates (If we choose a large $\zeta>0$ for which $V_0+\zeta V_1 \ge 0$,then the entire discrete spectrum of $H(z)$ disappears). On the other hand as $z$ varies along the negative real axis, $\kappa(z)$ moves further to negative side and remains as analytic function since its derivative $\int_{0}^{\infty}{V_1 \phi^2(z,x)}dx > 0$, $\phi(z,x)$ is the normalized real eigenfunction corresponding to $\kappa(z)$.

As z varies along the imaginary axis starting from $0$, then the discrete eigenvalues start moving along/opposite to the imaginary axis direction as analytic functions are conformal wherever derivative is non-zero. More precisely, if $V_1 \ge 0$, then as $z$ moves in the positive imaginary axis, $\kappa(z)$ also moves with tangent along the positive imaginary axis. For small values of imaginary $z=it$, the eigenvalue can be approximated as $\kappa(z)=\kappa(0)+i t \int_{0}^{a}{V_1 \phi^2(z,x)}dx$, where $[0,a]$ is the support of $V_j, j=0,1$.

In general, as $z$ varies over a continuous path in the complex plane, the eigenvalue $\kappa(z)$ also moves continuously until it ceases to be an eigenvalue. The function $\kappa(z)$ is analytic except at those points where it meets one or more such functions determined by the discrete spectrum of $H(z)$ or in other words the algebraic multiplicity of $\kappa(z)$ exceeds one. The following lemma characterizes this situation and it is  an elementary result.

\begin{lemma}\label{Ch2:Lemma:Multiplicity}
All the discrete eigenvalues of $H(z)$ are of geometric multiplicity one. And an eigenvalue $\kappa(z)=-\lambda^2(z)$ with eigenfunction $\phi(z,x)$ is not simple \emph{if and only if} $$\int_{0}^{\infty}{\phi^2(z,x)dx}=0$$.
\begin{proof}
First it is observed that if $\kappa(z)=-{\lambda^2}(z)$, with $\text{Re}(\lambda(z))>0$, is an eigenvalue of $H(z)$ and the support of $V(z)=V_0+zV_1$ is $[0,a]$, then the corresponding eigenfunction $\phi(z,x)$ satisfies:
\begin{equation}\label{generalEignFun}
\phi(z,x)=b~{\text{e}^{-\lambda(z)x}}, ~~~ \text{for} ~~ x \ge a
\end{equation}
for some non-zero $b$.

Let $\phi_1(z,x)$, $\phi_2(z,x)$ be two eigenfunctions corresponding to $\kappa(z)$, then there exist non-zero $b_1$ and $b_2$ such that, $\phi_1(z,x)=b_1 {\text{e}^{-\lambda(z)x}}$ and $\phi_2(z,x)=b_2 {\text{e}^{-\lambda(z)x}}$ for $x\ge{a}$. Thus $\phi(z,x)=b_2 \phi_1(z,x)-b_1 \phi_2(z,x)$ is the unique solution of the differential equation
\begin{equation*}
{-{{d^2\phi(z,x)}\over{dx^2}}}+V(z)\phi(z,x)=-\lambda^2(z)\phi(z,x)
\end{equation*}
on $[0,a]$ satisfying the condition $\phi(z,a)=0$ and $\phi'(z,a)=0$, prime denotes derivative with respect to $x$. Thus $\phi(z,x)=0$ or the geometric multiplicity of $\kappa(z)$ is one.

Now suppose $\kappa(z)=-\lambda^2(z)$ is not simple. Then $(H(z)+\lambda^2(z))^2\psi(z,x)=0$ and $(H(z)+\lambda^2(z))\psi(z,x)\ne0$ for some $\psi(z,x) \ne 0$ in $Dom(H(z))$. Since the geometric multiplicity of $-\lambda^2(z)$ is one,
\begin{equation}\label{multiEqn1}
(H(z)+\lambda^2(z))\psi(z,x)=c \phi(z,x)
\end{equation}
for some $c \ne 0$, where $\phi(z,x)$ is the normalized eigenfunction of $H(z)$ corresponding to $\kappa(z)=-\lambda^2(z)$. So we also have
\begin{equation}\label{multiEqn2}
(H(z)+\lambda^2(z))\phi(z,x)=0
\end{equation}
From (~\ref{multiEqn1}) and (~\ref{multiEqn2}),
\begin{equation*}
c \int_{0}^{\infty}{\phi^2(z,x)}dx=\int_{0}^{\infty}{{d\over{dx}}(\psi(z,x) \phi'(z,x)- \phi(z,x) \psi'(z,x))}dx=0
\end{equation*}
~\\
Conversely assume that $$\int_0^{\infty}\phi^2(z,x)=0$$
Since $\phi(z,x)=b \text{e}^{-\lambda(z)x}$ is on $[a,\infty)$, the function $\psi(z,x)={1 \over {2 \lambda(z)}}b x\text{e}^{-\lambda(z)x}$ is a solution of $(H(z)+\lambda^2(z))\psi(z,x)=\phi(z,x)$ on $[a,\infty)$. Extend the function $\psi(z,x)$ as a unique solution of $(H(z)+\lambda^2(z))\psi(z,x)=\phi(z,x)$ on $[0,a]$ satisfying the  conditions, which makes $\psi$ and $\psi'$ continuous at $x=a$. Thus we have a function $\psi(z,x)$ on $[0,\infty)$ such that $\psi'(z,x)$, $\psi''(z,x)$ are in $L^2(0,\infty)$ and it satisfies Equation~\ref{multiEqn1}. Repeating the same process as before we arrive at
\begin{equation*}
\int_{0}^{\infty}{{d\over{dx}}(\psi(z,x) \phi'(z,x)- \phi(z,x) \psi'(z,x))}dx=\int_{0}^{\infty}{\phi^2(z,x)}dx=0
\end{equation*}
\begin{equation*}
\implies \psi(z,a) \phi'(z,a)=0 \implies \psi(z,a)=0
\end{equation*}
This proves the existence of a function $\psi$ in the domain of $H(z)$, such that $(H(z)+\lambda^2(z))\psi(z,x) \ne 0$ and $(H(z)+\lambda^2(z))^2\psi(z,x) = 0$.
\end{proof}
\end{lemma}

\setcounter{theorem}{0}
Next it is shown that if the curve traced by the discrete eigenvalue $\kappa(z)$ of $H(z)$ as $z$ traces a curve in the complex plane terminates, then it terminates at the essential spectrum $[0, \infty)$ of $H(z)$.

\begin{theorem}\label{theorem1}
Let $z=\gamma(t)$ be a path in $\mathbb{C}$ and let $\kappa_0$ be an eigenvalue of $H(\gamma(0))$. As $z$ moves along the path $\gamma(t)$ starting from $\gamma(0)$, the eigenvalue traces a continuous path in $\mathbb{C}$, say $\kappa(\gamma(t))$ starting from $\kappa_0$. Assume that this path terminates at $t_1$ and let $\gamma(t_1)=\zeta$. Then $\kappa(\zeta)=\lim\limits_{t \rightarrow t_1-}{\kappa(\gamma(t))} \ge 0 $
\end{theorem}

\begin{proof}
Let us take $\kappa(z)=-\lambda^2(z)$. Since $z \rightarrow \zeta$, through a path in $\mathbb{C}$, we can find a sequence ${z_n}$ in the path with 
$z_n \rightarrow \zeta$ and hence $-\lambda^2(z_n) \rightarrow -\lambda^2(\zeta)$. Since $-\lambda^2(z_n)$ is an eigenvalue of $H(z_n)$, 
we have $\text{Re}(\lambda_n(z))>0 ~ \Rightarrow \text{Re}(\lambda(\zeta)) \ge 0$. Assume that $\text{Re}(\lambda(\zeta))>0$, we will derive a contradiction.\\
For $x \ge a$, $H(z_n)+\lambda^2(z_n)=-{\frac{d^2}{dx^2}} + \lambda^2(z_n)$ and hence the corresponding eigenfunction is 
$\phi(z_n,x)=b(z_n)\text{e}^{-\lambda(z_n)x}$. Without loss of generality we choose $b(z_n)=1$.
Thus for $x \ge a$, $\phi(z_n,x)=\text{e}^{-\lambda(z_n)x} \rightarrow \text{e}^{-\lambda(\zeta)x}$ uniformly.

On $[0,a]$, $\phi(z_n,x)$ is the solution of the differential equation 
$$-{\frac{d^2\phi(z_n,x)} {dx^2~~~~}} + (V_0+z_n V_1+\lambda^2(z_n)) \phi(z_n,x)=0$$ 
satisfying the conditions $$\phi(z_n,a)=\text{e}^{-\lambda(z_n)a} ~~\text{and}~~ \phi'(z_n,a)= -\lambda(z_n) \text{e}^{-\lambda(z_n)a}$$
Therefore using Lemma~\ref{lemma3}, $\phi(z_n,x) \rightarrow \phi(\zeta,x)$ uniformly on $[0,a]$, where $\phi(\zeta,x)$ 
satisfies the differential equation 
\begin{equation}\label{diffEqn1}
{-d^2\phi(\zeta,x) \over dx^2} + (V_0 + \zeta V_1) \phi(\zeta,x)=-\lambda^2(\zeta)\phi(\zeta,x)
\end{equation}
 on $[0,a]$ and $\phi(\zeta,a)=\text{e}^{-\lambda(\zeta)a}$, $\phi'(\zeta,a)= -\lambda(\zeta) \text{e}^{-\lambda(\zeta)a}$. 
Since $\phi(z_n,0)=0$ for all $n$, $\phi(\zeta,0)=0$. Thus we have proved that $\phi(z_n,x) \rightarrow \phi(\zeta,x)$ in $L^2(0,\infty)$, 
$\phi(\zeta,0)=0$ and satisfy the differential Equation~\ref{diffEqn1}. And hence $-\lambda^2(\zeta)$ is an eigenvalue of $H(\zeta)$, 
a contradiction.
\end{proof}

If $\kappa_1$ is a discrete spectrum member of $H(i)=-{\frac{d^2}{dx^2}}+V_0 + i V_1$. As $z:=it$ moves from $i$ to $0$ (that is, $t$ from $1$ to $0$) along the imaginary axis, $\kappa_1$ evolves continuously (analytically except for those points mentioned in Lemma~\ref{Ch2:Lemma:Multiplicity}) to trace a path $\kappa(t)$ 
in $\mathbb{C}$ and the above result ensures that either of the two possibilities occur:
\begin{enumerate}
\item{$\kappa(t)$ reaches the negative real line at $\kappa(0)=\kappa_0$, a discrete eigenvalue of the self-adjoint operator 
$H(0)=-{\frac{d^2}{dx^2}}+V_0$.}
\item{$\kappa(t)$ reaches $[0, \infty)$ at $\kappa_0$, a spectral singularity of $H(it_0)$, $0 \leq t_0 < 1$.}
\end{enumerate}
This proves the statement of Theorem~\ref{labMainTheorem} for compact potentials. But for compactly supported potentials a more general result is possible.

We have the following definitions (\cite{Abramov2001}):
\begin{defn}
Let $\kappa=-\lambda^2$ be a complex number and there exists a function $\phi(x)$ which satisfies the following conditions:

$$-\phi''(x)+V(x)\phi(x)=-\lambda^2 \phi(x), ~~ \text{on}~ [0,\infty)$$
$$\phi(0)=0$$
$$\phi(x)=e^{-\lambda x}+o(|e^{\lambda x}|), ~~ \text{as} ~ x\rightarrow \infty$$

Then if $\text{Re}(\lambda)<0$, $\kappa$ is a \textbf{resonance} of the operator $H=-{\frac{d^2}{dx^2}}+V$ defined on the domain 
$\{f \in L^2(0,\infty): f', f'', Vf \in L^2(0,\infty) ~ \text{and} ~ f(0)=0 \}$. If $\text{Re}(\lambda)=0$, then $\kappa$ is a 
\textbf{spectral singularity}.
\end{defn}
For compactly supported potentials, a perturbation of the potential gives rise to a same order of variation in the resonances 
of the Schr\"odinger operator (\cite{Agmon1998}). That is to say that the resonances of $H(z)$ move continuously with respect to $z$, provided $H(z)$ is not the free Schr\"odinger operator for any $z$ (see Section~\ref{Ch2:Sec:Example}). 
Thus we have the following:

\begin{corollory}
Let $V_0 \neq 0$, then for any discrete eigenvalue $\kappa_1$ of $H(i)$, there exists a discrete eigenvalue, or a spectral singularity or a resonance $\kappa_0$ of 
the self-adjoint operator $H(0)$ such that $\kappa_1$ is continuously evolved from $\kappa_0$ as $z$ varies along the imaginary line 
from $0$ to $i$. That is, there exists a continuous path $\kappa(t)$ such that $\kappa(0)=\kappa_0$, $\kappa(1)=\kappa_1$ and each 
$\kappa(t)$ is a discrete eigenvalue, or a spectral singularity or a resonance of the operator $H(it)$.
\end{corollory}

In fact, if $U\neq 0$, $V\neq 0$ are two compactly supported, complex continuous (except for jump discontinuity) functions on $[0,\infty)$, then any discrete eigenvalue of $H_V=-{\frac{d^2}{dx^2}}+V$ is evolved from a discrete eigenvalue, spectral singularity, or a resonance of $H_U=-{\frac{d^2}{dx^2}}+U$. This follows immediately from a similar analysis on the operator valued analytic function  $H(z)=-{\frac{d^2}{dx^2}}+U+z (V-U)$. If $V$ is a multiple of $U$, then it can happen that $U+z(V-U)=0$ for some $0\leq z \leq 1$. In this situation consider a different path so that the potential changes from $U$ to $V$ without taking $0$ on that path (see Section~\ref{Ch2:Sec:Example}).


\begin{corollory}
Let $M=\max{\{|V_1(x)|: x \in [0,a]\}}$, $\kappa_0$ be a discrete eigenvalue of the self-adjoint operator $H(0)=-{\frac{d^2}{dx^2}}+V_0$. Let $\gamma$ be a path in $\mathbb{R}$, say $\gamma(t)=t$, and $\kappa(t)$ be the path traced by discrete eigenvalue or spectral singularity or resonance of $H(t)$ with $\kappa(0)=\kappa_0$ as $t$ varies over the real line starting from $0$. Then for $|t|<|\kappa_0|/M$,  $\kappa(t)$ remains to be a discrete eigenvalue of $H(t)$. In particular, if $\kappa_0$ is the discrete eigenvalue of $H(0)$ nearest to $0$ then $$|\sigma_\text{d}(H(t))| \ge |\sigma_\text{d}(H(0))|$$ for real $t$ with $t < |\kappa_0|/M$.
\end{corollory}
\begin{proof}
As $t$ varies from $0$ through the real line, the discrete eigenvalue $\kappa(t)$ starts from the negative real number $\kappa_0$ and moves analytically as a discrete eigenvalue of $H(t)$ until it reaches $0$. Assume that at $t=t_0$ it reaches $0$. Then
$$0=\kappa_0+\lim\limits_{t \rightarrow t_0}\int_0^{t}{\int_0^{a}{V_1(x) \phi^2(s,x)} dx}~ds$$ 
$$\Rightarrow |\kappa_0| \le M |t_0| = M |t_0|{~~~~~~~~~~~~~~~~~~~~~~~~~~~~~}$$ 
here $\phi(s,\cdot)$ represents the normalized eigenfunction of $H(s)$ corresponding to $\kappa(s)$.

Hence $\kappa(t)$ remains to be an eigenvalue of $H(t)$, if $t \in \mathbb{R}$ is such that $|t|< |\kappa_0|/{M}$.

It immediately follows that if $\kappa_0$ is the discrete eigenvalue of $H(0)$ nearest to $0$, then for real $t$ with 
$|t|<|\kappa_0| / {M}$,$$ |\sigma_{\text{d}}(H(t))| \ge |\sigma_{\text{d}}(H(0))|$$
\end{proof}

It is been observed in the beginning of this section that as $z$ moves along the imaginary line starting from $0$, each of the discrete eigenvalues of $H(z)$ starts moving in or opposite to the direction of imaginary axis. So one would expect, in general, a better estimate than the previous result.
\begin{conj}
If $V_1>0$ (or $V_1<0$) on $[0,a]$, the support of $V_0$ and $V_1$, then each of the discrete eigenvalue of the self-adjoint operator $H(0)$ continuously evolves to a 
discrete eigenvalue of $H(it)$ for any $t \in \mathbb{R}$. In particular 
$$\left|\sigma_{\text{d}}\left( -{\frac{d^2}{dx^2}} + V_0 + i V_1\right)\right| \ge \left|\sigma_{\text{d}}\left( -{\frac{d^2}{dx^2}} + V_0 \right)\right|$$ 
where the discrete eigenvalues are counted according to multiplicity.
\end{conj}

\section{A More General Case}\label{Ch2:Sec:GeneralResult}
In this section we assume that $V_0$, $V_1$ are bounded real measurable functions satisfying 
\begin{equation}\label{eqnPolDecPotential}
\int_0^{\infty}{x |V_j(x)| dx} < \infty,~~~~j=0,1
\end{equation}
Here the operator $H(z)=-{\frac{d^2}{dx^2}}+V_0+z V_1$ is a compact perturbation of the free Laplacian and hence its spectrum consists 
of essential spectrum $\sigma_{ess}=[0,\infty)$ and a countable number of discrete eigenvalues which can only accumulate to a point in the 
essential spectrum $[0,\infty)$. The spectral analysis of Schr\"odinger operators on the half-line with potential satisfying condition (\ref{eqnPolDecPotential}) has been carried out by several authors (\cite{Pavlov1967} and references therein). Our approach is similar to the one in the previous section.

We start with presenting information required for our analysis. It is known that if a complex potential $V(x)$ has the property stated in (\ref{eqnPolDecPotential}), the equation 
\begin{equation}
-y''+V(x) y= -\lambda^2 y, ~~~~ \text{Re}(\lambda) \ge 0
\end{equation}
has a unique solution $\phi(\lambda,x)$ satisfying the condition $\phi(\lambda,x)~\text{e}^{\lambda x} \rightarrow 1$ as $x \rightarrow \infty$. 
The function $\phi(\lambda,x)$ also satisfies (\cite{Agranovich1959,Pavlov1967}) the following estimates
\begin{equation}\label{eqnUniformBound}
|\phi(\lambda,x)-\text{e}^{-\lambda x}| \le K~ \text{e}^{-\text{Re}(\lambda) x} \int_{x}^{\infty}{t |V(t)| dt}, ~~~ \text{Re}(\lambda) \ge 0,
\end{equation}
\begin{equation}\label{eqnEquicontinuity}
|\phi_x(\lambda,x)+\lambda~\text{e}^{-\lambda x}| \le K~ \text{e}^{-\text{Re}(\lambda) x} \int_{x}^{\infty}{|V(t)| dt}, ~~~ \text{Re}(\lambda) \ge 0.
\end{equation}
Now we prove Theorem~\ref{labMainTheorem}.
\begin{proof}
It is sufficient to prove the statement of Theorem~\ref{theorem1} in this more general case of potentials. That is, if a sequence 
$z_n \rightarrow \zeta$ in $\mathbb{C}$ and discrete eigenvalue $\kappa(z_n)=-\lambda^2(z_n)$ of $H(z_n)$ converges to $\kappa(\zeta)=-\lambda^2(\zeta)$, 
then $\kappa(\zeta)$ is either a discrete eigenvalue or a spectral singularity of the operator $H(\zeta)$.

Suppose $\phi(z_n,x)$ be the corresponding unique eigenfunction that satisfies Equations~\ref{eqnUniformBound}, ~\ref{eqnEquicontinuity} and $\phi(z_n,x)\text{e}^{\lambda(z_n) x} \rightarrow 1$. Since 
$|\phi(z_n,x)|\le |\phi(z_n,x)-\text{e}^{-\lambda(z_n) x}| + \text{e}^{-\text{Re}(\lambda(z_n))x}$ and $\text{Re}(\lambda(z_n)) > 0$, it follows from the 
estimate (\ref{eqnUniformBound}) that $\{\phi(z_n,x)\}$ is a uniformly bounded sequence. In similar lines estimate (\ref{eqnEquicontinuity}) 
ensures that $\{\phi_x(z_n,x)\}$ is uniformly bounded or the sequence $\{\phi(z_n,x)\}$ is equicontinuous. Therefore by Arzel\`a-Ascoli theorem 
$\phi(z_n,x) \rightarrow \phi(\zeta,x)$ uniformly on any compact subset of $[0,\infty)$ and we will have 
$\phi(\zeta,x)\text{e}^{\lambda(\zeta) x} \rightarrow 1$, $-\phi''(\zeta,x)+(V_0(x) + \zeta V_1(x))\phi(\zeta,x)=-\lambda^2(\zeta)\phi(\zeta,x)$, 
$\phi(\zeta,0)=0$. Thus if $\text{Re}(\lambda(\zeta))>0$ then $\kappa(\zeta)$ is an eigenvalue otherwise (that is, if $\text{Re}(\lambda(\zeta))=0$) it is a spectral singularity.
\end{proof}
The following result can be proved the same way as it is done in the previous section.
\begin{corollory}
Let $M=\sup{\{|V_1(x)|: x \in [0,\infty)\}}$, $\kappa_0$ be a discrete eigenvalue of the self-adjoint operator $H(0)=-{\frac{d^2}{dx^2}}+V_0$. Let $\gamma$ be a path in $\mathbb{R}$, say $\gamma(t)=t$, and $\kappa(t)$ be the path starting at $\kappa_0$ traced by discrete eigenvalues, or spectral singularities or resonances of $H(t)$ as $t$ varies over the real line starting from $0$. Then for $|t| < |\kappa_0|/{M}$, $\kappa(t)$ remains to be a discrete eigenvalue of $H(t)$. In particular, if $\kappa_0$ is the discrete eigenvalue of $H(0)$ nearest to $0$ then 
$$|\sigma_{\text{d}}(H(t))| \ge |\sigma_{\text{d}}(H(0))|$$ for real $t$ with $t < |\kappa_0|/{M}$.
\end{corollory}

\section{An Example}\label{Ch2:Sec:Example}
In this section we demonstrate the continuous movement of resonances and discrete spectrum of Schr\"odinger operators with potentials that are constant on their support. Consider the potential $V$ defined on $[0,\infty)$ by
 \begin{equation} V=\left\{ 
\begin{array}{ll}
-k^2  & \text{on} ~ [0,1] \\
~~~0  & \text{elsewhere} \\
\end{array}
\right.
\label{Eqn:Potential}
\end{equation}
Let $\kappa=-\lambda^2$ be an eigenvalue of the Schr\"odinger operator $H=-\frac{d^2}{dx^2}+V$ with domain $\text{Dom}(H)=\{f: f,f',f'' \in L^2(0,\infty), f(0)=0\}$, then there exists $\phi$ such that $H \phi = -\lambda^2 \phi$, $\phi(0)=0$, and $\phi,~\phi',~\phi'' \in L^2(0,\infty)$. All these conditions imply that $\text{Re}(\lambda) > 0$ and
\begin{equation}
f(\lambda)=\lambda \sin(\sqrt{k^2-\lambda^2}) + \sqrt{k^2-\lambda^2} cos(\sqrt{k^2-\lambda^2})=0
\label{Eqn:Characteristic}
\end{equation}
which is the characteristic equation for the eigenvalue problem of the given operator. Also note that the spectral singularities and resonances of the given Schr\"{o}dinger operator are also obtained from the above equation. They are  $\kappa=-\lambda^2$ where $\lambda$ satisfies the characteristic equation (other than $\pm k$) with $\text{Re}(\lambda)=0$ (spectral singularity)and $\text{Re}(\lambda)<0$ (resonance).

If the value of $V$ on $[0,1]$ is real, then the operator is a self-adjoint operator and all its eigenvalues are real and complex resonances exist in symmetric pairs. Consider a complex value $-k^2 + \zeta(t)$ on $[0,1]$ for the potential $V$, where $\zeta(t)$ is a continuous path in $\mathbb{C}$. By Theorem~\ref{theorem1}, as $t$ moves over the real line the eigenvalues or resonances of the operator move continuously in the complex plane and an eigenvalue moves to the resonance set through the positive real axis $[0,\infty)$ (which is the essential spectrum of the operator) and vice versa.

In particular, if we assume the value $-k^2+t$ ($k^2$ is real) for $V$ on $[0,1]$, as $t$ varies over the real line the eigenvalues analytically move over the negative real axis until they touch $0$ and move to the resonance set. A complex pair of symmetric resonances traces a pair of symmetric curves until they meet at the real axis and at this meeting point $f'(\lambda)=0$. That is
\begin{equation}
(\lambda+1)\left[ \sin(\sqrt{k^2-\lambda^2})-\frac{\lambda}{\sqrt{k^2-\lambda^2}} \cos(\sqrt{k^2-\lambda^2}) \right]=0
\end{equation}
This gives either $\lambda=-1$ or
$$\sqrt{k^2-\lambda^2} \sin(\sqrt{k^2-\lambda^2}) - \lambda \cos(\sqrt{k^2-\lambda^2})=0$$
The above expression along with the characteristic equation $f(\lambda)=0$ implies that $k=0$ and $\exp(\lambda)=0$. Thus the symmetric complex resonances meet in $\mathbb{R}$ at $-\lambda^2=-1$ and at this point characteristic equation becomes
\begin{equation}
\tan(\sqrt{k^2-1})=\sqrt{k^2-1}
\end{equation}
\emph{That is}, a pair of complex resonances meet at $-1$ as $t$ varies over real line and at this meeting point the potential takes the value $-k^2+t=-(\theta^2+1)$ on $[0,1]$, where $\theta$ is a solution of the equation $\tan{\theta} = \theta$. Each interval $[ n \pi, (2 n + 1) \pi/2]$ for $n=0,1,2,\ldots$ contains a solution $\theta_n$ of $\tan{\theta}=\theta$. But as $-k^2+t \rightarrow -1=-(\theta_0^2+1)$ it can be seen that the only real resonance (antibound state) of the operator reaches $-1$. As $t$ moves further to the negative side, different complex symmetric pairs of resonances  move symmetrically with respect to the real axis and meet at $-1$ as the value of the potential on $[0,1]$ becomes $-k^2+t=-(\theta_n^2+1), n=1,2,\ldots$. This is shown in Figure~\ref{Fig:resonanceEvolution}. This figure is zoomed about $-1$ and is shown in Figure~\ref{Fig:resonanceEvolutionZoom}. These figures show the movement of few symmetric pairs of resonances as the potential takes the value $-0.5+t$ on $[0,1]$, $t$ varies from $0$ to the negative side of the real line. The movement of the symmetric curve is restricted in these figures and is shown up to their meeting point $-1$ on the real line. As $t$ moves further to the negative side, this symmetric pair of resonances once met at $-1$ keep moving in real line, one towards the negative side and the other towards the positive side. The resonance moving towards the positive side meets the positive real axis at $0$ and re-bounces back as an eigenvalue. This is illustrated in Figure~\ref{Fig:resonanceToeigenvalue}. The potential at which it meets the real axis is obtained from the characteristic equation by substituting $\lambda=0$ in it and the corresponding potential is $-[(2 n + 1) \pi / 2]^2$, for $n=0,1,2,\ldots$. Thus we have the following:

\begin{proposition}\label{proposition1}
Let $n \in \mathbb{N}$. If the potential $V$ is a real constant, say $-k^2$, on its support $[0,1]$ and $[(2n - 1)\pi/2]^2 < k^2 \leq [(2n + 1) \pi/2]^2$ then the above Schr\"{o}dinger operator has exactly $n$ eigenvalues. If $K_n=\theta_n^2 +1$ where $\theta_n$ is the root of $\tan{\theta}=\theta$ in the interval $[n \pi, (2n+1) \pi/2]$ then the above Schr\"{o}dinger operator has exactly $n-1$ antibound states if $[(2n-1)\pi/2]^2 < k^2<K_n$ and $n+1$ antibound states if $K_n\le k^2 \leq [(2n+1)\pi/2]^2 $.
\end{proposition}

\begin{theorem}\label{Ch2:Theorem:EigenvalueCount}
Let $V$ be a real potential with compact support $[0,1]$. Suppose that there exists $m < n$ such that $-[(2n - 1)\pi/2]^2 < V < -[(2m - 1) \pi/2]^2$ on $[0,1]$. Then $$m < |\sigma_{\text{d}}(H)| < n$$ where $H=-\frac{d^2}{dx^2}+V$ is the self-adjoint Schr\"odinger operator.
\end{theorem}
\begin{proof}
This is an immediate consequence of Theorem~\ref{theorem1} and Proposition~\ref{proposition1}.
\end{proof}

If the potential is constant on the support $[0,1]$, then \cite{Frank2016} provides the following estimate $$|\sigma_{\text{d}}(H)|<\frac{|V|^2}{\epsilon^2}\left({\frac{\text{e}^\epsilon-1}{\epsilon}}\right)^2$$ where $V$ is the constant value on the support. The minimum value of $$\frac{1}{\epsilon^2}\left({\frac{\text{e}^\epsilon-1}{\epsilon}}\right)^2$$ can be found out numerically, $\approx 2.38436418$. Thus the estimate reduces to $$|\sigma_{\text{d}}(H)|<2.38436418~|V|^2.$$ But if $V$ is real and constant, we can use the result for a self-adjoint operator and obtain a better estimate $$|\sigma_{\text{d}}(H)|<\int\limits_0^\infty{x V_{-}(x) dx}=\frac{|V|}{2}.$$ Whereas Theorem~\ref{Ch2:Theorem:EigenvalueCount} gives us $$|\sigma_{\text{d}}(H)|= \text{integer part of}\left(\frac{1}{\pi}\sqrt{|V|}+\frac{1}{2}\right).$$ Or for a potential $V$ which is continuous and bounded on its support $[0,1]$, $$ \frac{1}{\pi}\sqrt{\text{min}(V_{-})}+\frac{1}{2}<|\sigma_{\text{d}}(H)| < \frac{1}{\pi}\sqrt{\text{max}(V_{-})}+\frac{1}{2}$$ where $\min(V_{-}), \max(V_{-})$ are minimum and maximum of the negative part of $V$ on its support $[0,1]$. Thus it is easy to see that if the variation of potentials on its support is minimal then Theorem~\ref{Ch2:Theorem:EigenvalueCount} provides better estimate.

\begin{figure}
\centering
\includegraphics[width=0.9\textwidth]{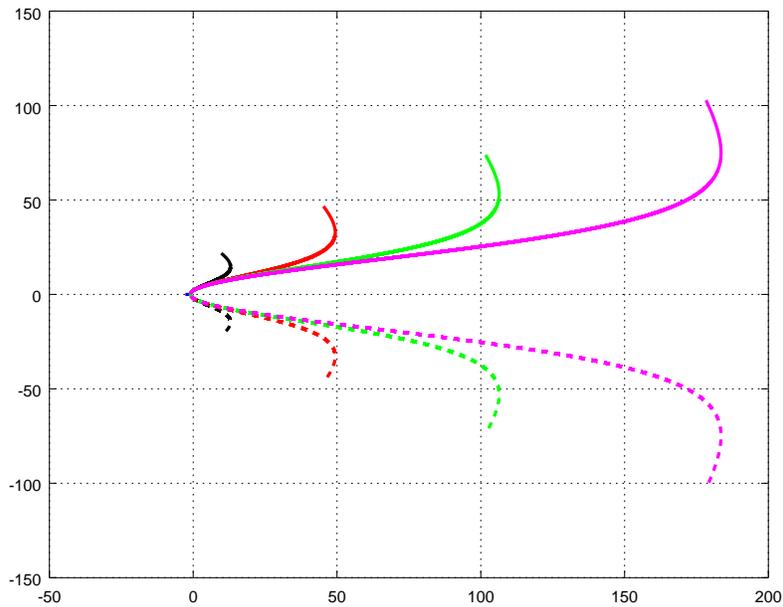}
\caption{Evolution of resonances of the given Schr\"{o}dinger operator. The potential starts at $-0.5$ and decreases further. The blue curve represents the real resonace which moves right and becomes $-1$ at $k=-1=-(\theta_0^2 +1)$, $\theta_0=0$ is the first non-negative solution of $\tan{\theta}=\theta$. The black, red, green and magenta curves are the evolution of the four sets of symmetric resonances close to the origin. Each of these meet the real axis at $-1$ at potentials $-(\theta_1^2+1), -(\theta_2^2+1), -(\theta_3^2+1)$ and $-(\theta_4^2 +1)$  respectively where $\theta_n$ is the solution of $\tan{\theta}=\theta$ in the interval $[n \pi, (2n+1) \pi/2]$, $n=1, 2, 3, 4$ }
\label{Fig:resonanceEvolution}
\end{figure}

\begin{figure}
\centering
\includegraphics[width=0.9\textwidth]{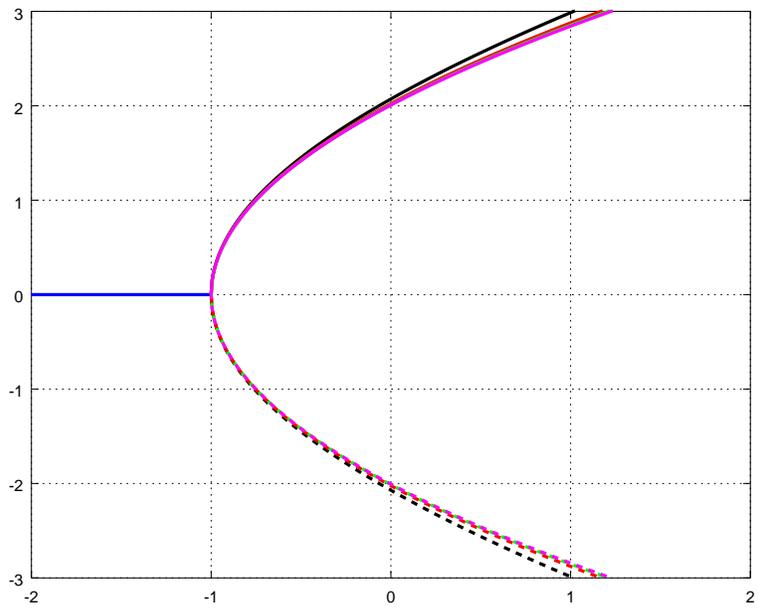}
\caption{A zoomed version of the Figure~\ref{Fig:resonanceEvolution} about the origin.}
\label{Fig:resonanceEvolutionZoom}
\end{figure}
\begin{figure}
\centering
\includegraphics[width=0.9\textwidth]{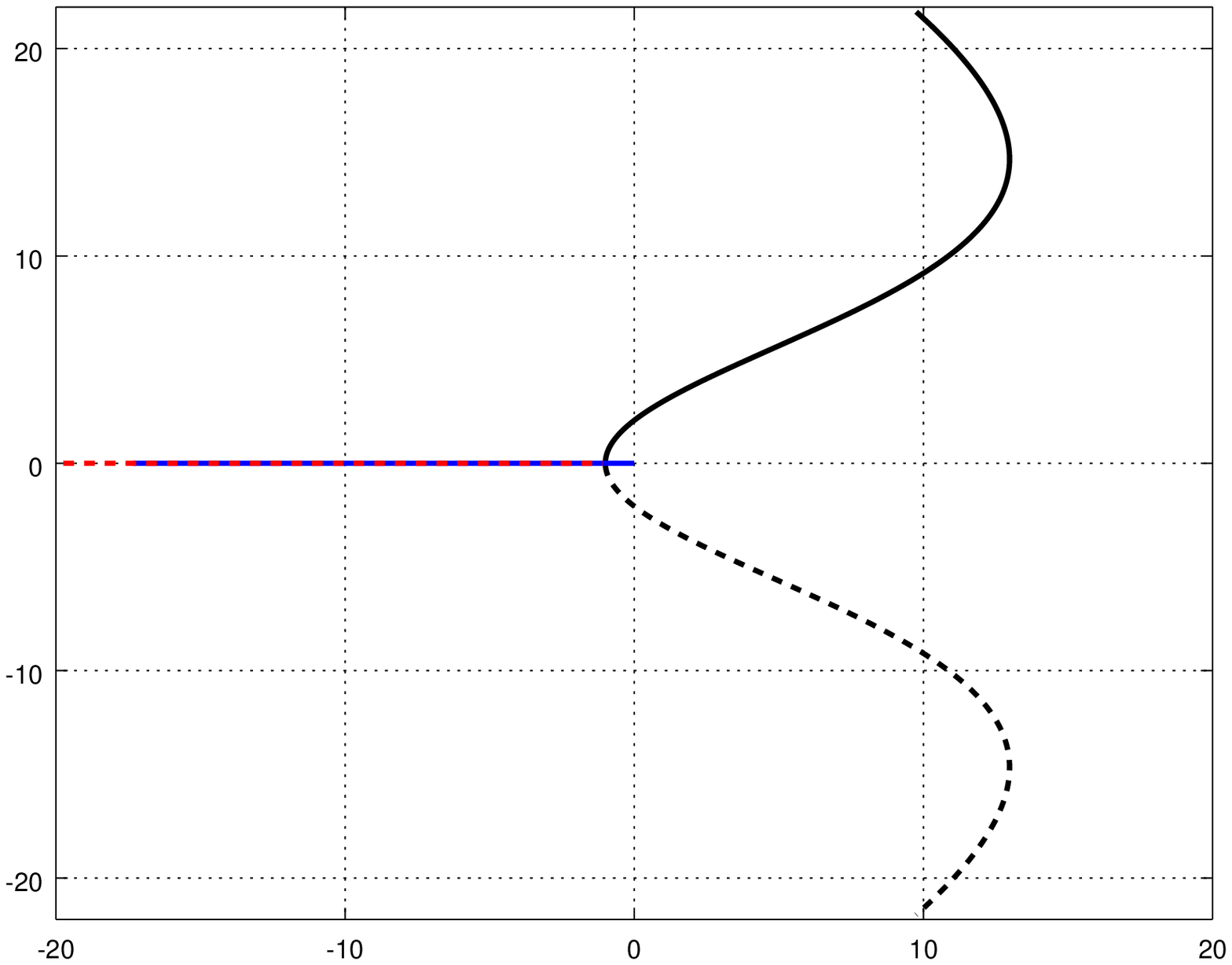}
\caption{Evolution of resonance into eigenvalue of the Schr\"{o}dinger operator. The black curves (solid and dashed) indicate the evolution of two complex symmetric resonances that meet at $-1$ as the real potential decreases to $-(\theta^2+1)$, where $\theta$ satisfies $\tan{\theta}=\theta$. Then they separate out and move in opposite directions along the real axis. The one (blue curve) moving to the positive direction meets the positive real axis at $0$, as potential takes the value $-[(2n+1) \pi/2]^2$and bounces back as an eigenvalue which moves further to the negative side as potential decreases further. The other (red dashed line) that moves to the negative direction continues as a real resonance (antibound state).}
\label{Fig:resonanceToeigenvalue}
\end{figure}

Now we go back to the example and demonstrate the evolution of eigenvalues or resonances as the imaginary part of $-k^2$ continuously moves and traces vertical lines in the complex plane. For an example, the value of the potential on $[0,1]$ is initially taken as $-22$. By Proposition~\ref{proposition1}, the corresponding self-adjoint Schr\"odinger operator has one eigenvalue and two antibound states on the real line. Consider these bound and antibound states and the symmetric pair of complex resonances that are close to the origin. The evolution of these eigenvalues and resonances of this operator as $-k^2$ varies from $-22$ to $-22+250i$ or $-22-250i$ are shown in Figure~\ref{Fig:evolution}. It is observed that the eigenvalue of the self-adjoint operator remains to be an eigenvalue as the imaginary part of $-k^2$ varies from $0$ to the positive or negative side. The same thing happens to one of the real resonances (which is less than $-1$). Whereas the other resonance (on the right side of $-1$) traces a curve which crosses $[0,\infty)$ and changes to an eigenvalue as the imaginary part varies from $0$ to $250$ or $-250$ (see Figure~\ref{Fig:evolutionZoom} which is a zoomed version of Figure~\ref{Fig:evolution} about the origin). One of the complex resonances changes its status to eigenvalue and the other remains to be a resonance as the imaginary part of $-k^2$ changes from $0$ to the positive or negative side.
\begin{figure}
\centering
\includegraphics[width=0.9\textwidth]{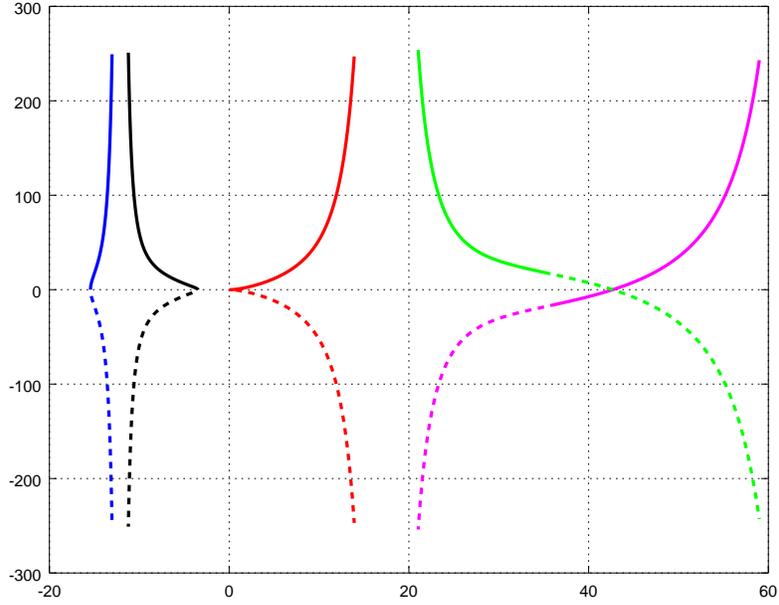}
\caption{The evolution of the negative eigenvalue ($\approx-15.42901680$), two antibound states ($\approx -0.01187978, -3.48239885$) and a pair of complex resonances ($\approx 35.73924059 + 16.82276560i, 35.73924059 - 16.82276560i$) of the self-adjoint operator with potential equals to $-22$ on its support $[0,1]$ as the imaginary part of the potential varies from $0$ to $250$ and to $-250$ on the support $[0,1]$. The dashed curve corresponds to the variation from $0$ to $-250$. The eigenvalue (blue curve) remains to be an eigenvalue, one of the negative resonance (black curve) remains as a resonance while the other resonance (red curve) crosses $[0,\infty)$ and moves to discrete spectrum. One resonance in the pair of complex resonances crosses $[0,\infty)$ and moves to the discrete spectrum while the other remains to be a resonance as the imaginary part varies from $0$ to the positive side or to the negative side.}
\label{Fig:evolution}
\end{figure}

\begin{figure}
\centering
\includegraphics[width=0.9\textwidth]{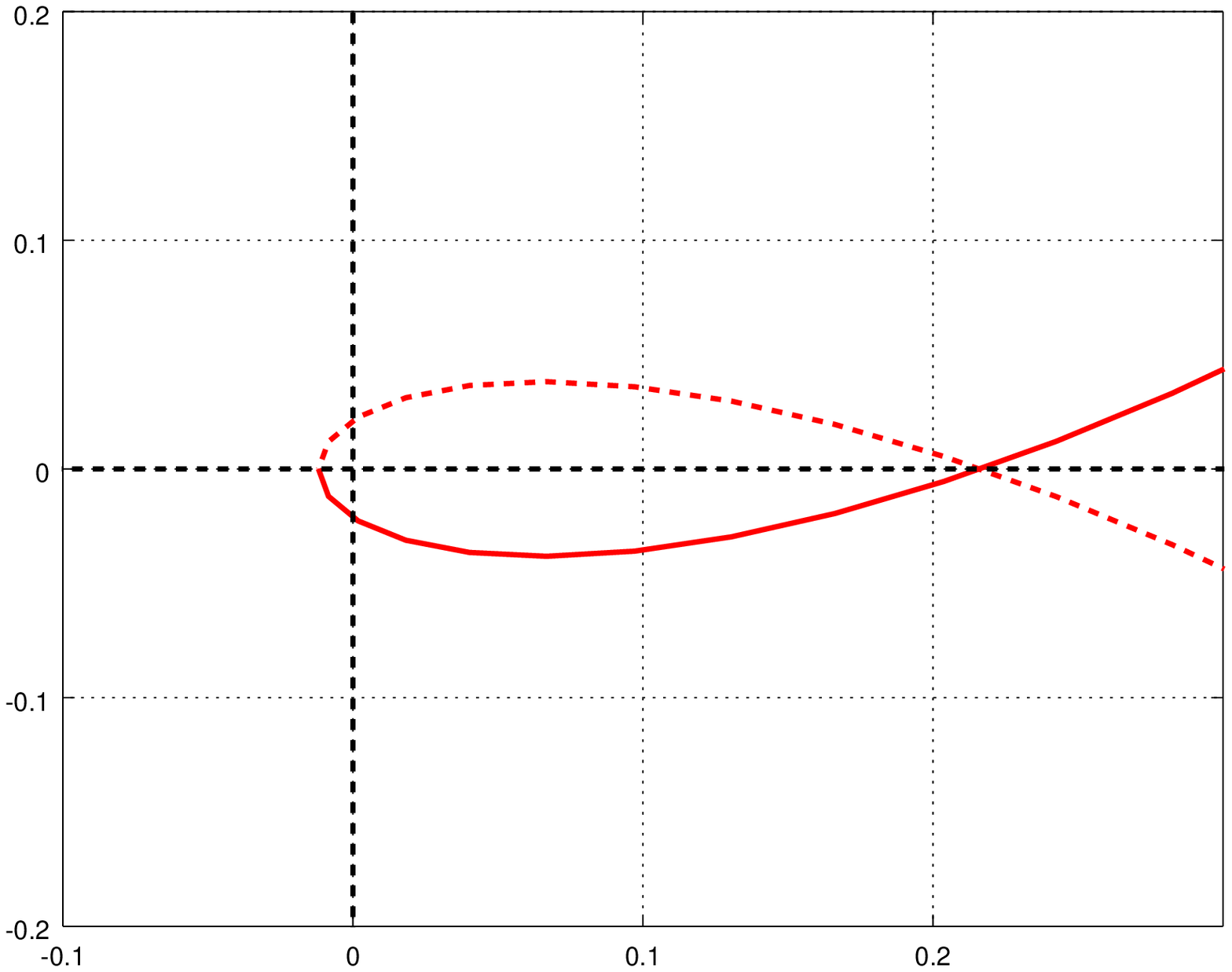}
\caption{A zoomed version of Figure~\ref{Fig:evolution} about the origin. The evolution of the negative resonance on the right side of $-1$ as the imaginary part varies from $0$ to positive or negative values. The resonance crosses $[0,\infty)$ and moves to the discrete spectrum of the corresponding Schr\"odinger operator.}
\label{Fig:evolutionZoom}
\end{figure}

\paragraph{Finding resonance for positive potential.} It is evident from Equation~\ref{Eqn:Characteristic} that if the value of $k$ is zero or for the free Schr\"odinger operator, there is no eigenvalue or resonance in the complex plane. As the value of $-k^2$ varies from nonzero to zero, all the eigenvalues and resonances of the operator diverge to $\infty$. Thus to find evolution of resonances or eigenvalues of Schr\"odinger operators as the potential changes from $U\neq 0$ to another $V\neq 0$, one should choose a path so that for any $z$ in the path, $H(z)$ does not become the free Schr\"odinger operator. For example, consider the variation of $-k^2$ in Equation~\ref{Eqn:Potential} from $-0.5$ to $1$. If we choose the path along the real line, then at $0$ the corresponding operator reduces to free Schr\"odinger operator. Thus it is required to consider a different path to get the evolution of the resonances. For example, if we consider the polygonal path $-0.5\rightarrow -0.5+i \rightarrow 1+i \rightarrow 1$, then it may be possible to find the evolution of the resonances. The evolution of the antibound state of the operator with $-k^2=-0.5$ to a resonance of the operator with $-k^2=1$ along this path is shown in Figure~\ref{Fig:evolutionAntiboundState}.

\paragraph{Finding eigenvalue for potential with non-negative real part.} There does not exist any eigenvalue for self-adjoint Schr\"odinger operator with non-negative potential. But if an imaginary part is added to the potential then the corresponding nonself-adjoint Schr\"odinger operator can have discrete eigenvalues. If we want to find out these eigenvalues as an evolution of some negative eigenvalues of some self-adjoint Schr\"odinger operator, a similar idea explained in the above paragraph can be employed. Here first a sufficiently negative potential is considered and its eigenvalues are estimated. Then imaginary part is introduced continuously to the potential and finally the real part of the potential is increased to the required value. For example, consider the potential in Equation~\ref{Eqn:Potential} and suppose we want to find out one eigenvalue when $-k^2$ takes the value $10+5i$. Initially $-k^2$ is assumed a negative value of $-10$ then it is changed along the polygonal path $-10\rightarrow -10+5i \rightarrow 10+5i$. The evolution of the single discrete eigenvalue of the self-adjoint operator is shown in Figure~\ref{Fig:evolutionEigenvalue}.

\begin{figure}
\centering
\includegraphics[width=0.9\textwidth]{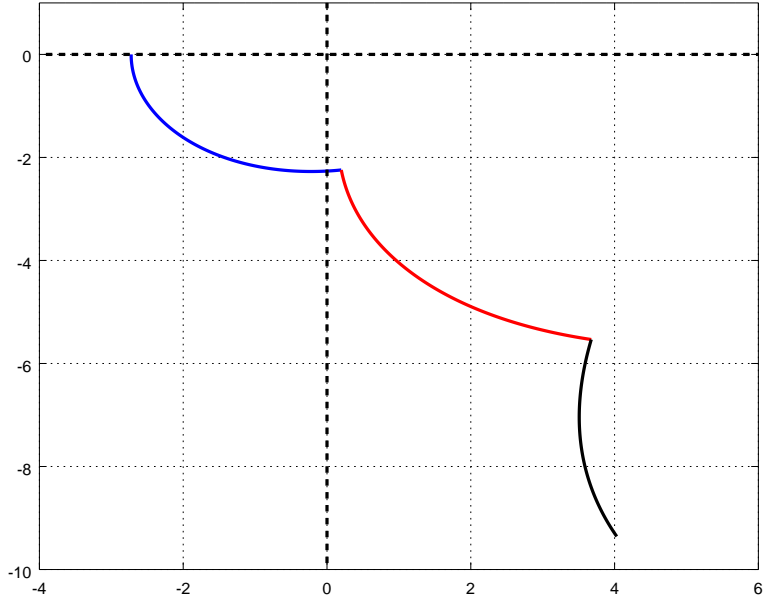}
\caption{Evolution of a resonance as the value of $-k^2$ varies from $-0.5$ to $-0.5+i$ (blue curve) then to $1+i$ along real line (red curve) and finally to $1$ along imaginary line. The antibound state ($\approx -1.65056781$) of the Schrodinger operator with potential equal to $-0.5$ on its support $[0,1]$ becomes a resonance ($\approx 4.02995187 - 9.35784001i$) of the the operator with potential equal to $1$ on its support.}
\label{Fig:evolutionAntiboundState}
\end{figure}

\begin{figure}
\centering
\includegraphics[width=0.9\textwidth]{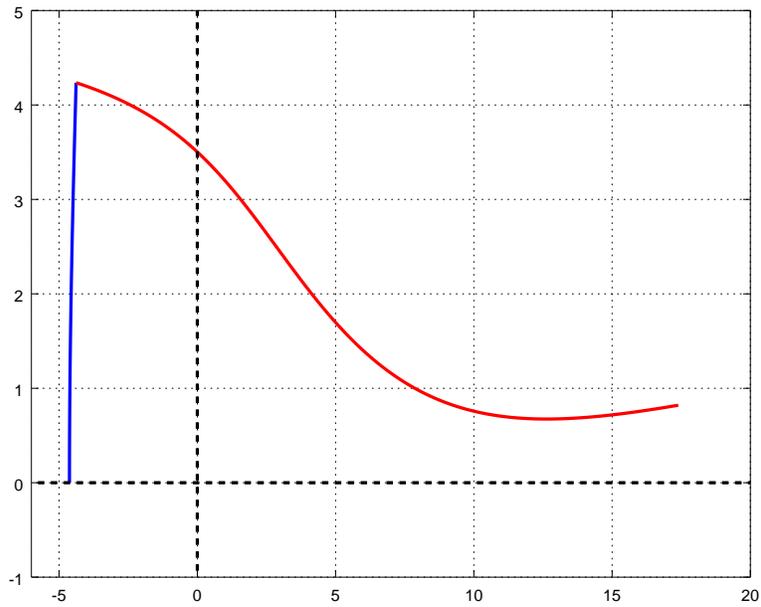}
\caption{Evolution of the negative eigenvalue of the self-adjoint Schr\"odinger operator as the value of $-k^2$ in Equation~\ref{Eqn:Potential} varies from $-10$ to $-10+5i$ and then to $10+5i$. Initially the discrete eigenvalue was $\approx -4.62419409$, it gets evolved and becomes $\approx 17.39128151 +  0.82067130i$ as $-k^2$ reaches the value of $10+5i$.}
\label{Fig:evolutionEigenvalue}
\end{figure}

\bibliographystyle{plainnat}
\bibliography{schrodinger}

\end{document}